\documentclass[a4paper,11pt]{amsart}

\usepackage[english]{babel}
\usepackage{amsmath,amsthm}
\usepackage{amsfonts}
\usepackage{enumerate}
\usepackage{graphicx}
\usepackage{color}
\usepackage[all]{xy}
\usepackage{thmtools, thm-restate}
\usepackage{overpic}
\usepackage{enumitem}
\usepackage{hyperref}

\newcommand{\blue}{\textcolor{blue}}

\usepackage[top=1.25in, bottom=1.25in, left=1.25in, right=1.25in]{geometry}

\newtheorem{thm}{Theorem}[section]
\newtheorem{cor}[thm]{Corollary}
\newtheorem{lem}[thm]{Lemma}
\newtheorem{prop}[thm]{Proposition}

\theoremstyle{definition}

\theoremstyle{remark}
\newtheorem{rem}[thm]{Remark}

\numberwithin{equation}{section}

\newcommand{\spin}{\ifmmode{\rm Spin}\else{${\rm spin}$\ }\fi}
\newcommand{\spinc}{\ifmmode{{\rm Spin}^c}\else{${\rm spin}^c$}\fi}

\newcommand{\Z}{\mathbb{Z}}

\newcommand{\Q}{\mathbb{Q}}

\newcommand{\R}{\mathbb{R}}

\newcommand{\n}{\mathbf n}

\newcommand{\xtup}{\mathbf x}

\DeclareMathOperator{\len}{len}

\begin{document}

\title{Non-simply connected symplectic fillings of lens spaces}%

\author{Paolo Aceto}%
\address {University of Oxford}
\email{paoloaceto@gmail.com }

\author{Duncan McCoy}%
\address {Universit\'{e} du Qu\'{e}bec \`{a} Montr\'{e}al}
\email{mc\_coy.duncan@uqam.ca}

\author{JungHwan Park}%
\address {Georgia Institute of Technology}
\email{junghwan.park@math.gatech.edu}
\date{\today}%

\begin{abstract}
We prove results exploring the relationship between the fundamental group and the second Betti number of minimal symplectic fillings of lens spaces. These results unify and generalize several disparate facts appearing in the literature. The Fibonacci numbers make a cameo appearance.
\end{abstract}
\maketitle

\section{Introduction}
Understanding the symplectic fillings of contact manifolds is a problem with a long history in contact and symplectic geometry. In this article, we focus on symplectic filling of lens spaces. The lens space $L(p,q)$ is the manifold obtained by $-p/q$-surgery on the unknot where $p > q > 0$ are relatively prime integers.

Eliashberg~\cite{Eliashberg:1990-1} proved that every symplectic filling for the standard tight contact structure on $S^3$ is obtained by a blowup of the standard symplectic $B^4$. Moreover, McDuff~\cite{McDuff:1990-1} classified the symplectic fillings for the standard tight contact structures on the lens spaces $L(p, 1)$.

Lisca~\cite{Lisca:2008-1} extended these results by classifing the symplectic fillings of the standard tight contact structures on every lens space. Moreover, recently Etnyre-Roy~\cite{Etnyre-Roy:2020-1} and Christian-Li \cite{Christian-Li:2020-1}  classified the symplectic fillings of all the tight contact structures on lens spaces. A symplectic filling $X$ is said to be minimal if it does not admit any symplectic blow-downs, that is it doesn't contain any embedded symplectic 2-spheres with self-intersection $-1$. Since every tight contact structure on a lens space is planar~\cite{Schonenberger}, every minimal symplectic filling is symplectically deformation equivalent to a Stein filling \cite{Wendl-Niederkruger}. So in this setting the study of Stein fillings and minimal symplectic fillings are effectively equivalent. In this paper, we find several bounds on the topology of minimal symplectic fillings of lens spaces. These results unify several preexisting results in the literature and the general theme is that a large fundamental group implies a small second Betti number.

\begin{thm}\label{thm:main2}
If $X$ is a minimal symplectic filling of $(L(p,q), \xi)$ with $|\pi_1(X)| =d$, then $d^2$ divides $p$ and 
\begin{equation}\label{eq:b2-bound}
b_2(X)\leq \frac{p}{d^2} -1.
\end{equation}
Moreover, if $\xi$ is virtually overtwisted then this inequality is strict.
\end{thm}

This extends an observation of Fossati, who showed that $d$ divides $p$ and that every symplectic filling of $L(p,q)$ is simply connected if $p$ is prime \cite[Corollary 14]{Fossati:2019-1} (see also \cite[Theorem~1.16]{Etnyre-Roy:2020-1}). Whereas Theorem~\ref{thm:main2} implies that every symplectic filling is simply connected if $p$ is square-free. Finally, Theorem~\ref{thm:main2} also recovers the fact that a virtually overtwisted contact structure on a lens space cannot be filled with a rational homology ball \cite[Proposition~A.1]{Golla-Starkston}, \cite[Lemma 1.5]{Etnyre-Roy:2020-1} (see also \cite[Theorem~4]{Fossati:2019-1}).

Furthermore the bound \eqref{eq:b2-bound} is sharp and one can characterize those lens spaces attaining equality. 
\begin{prop}\label{prop:main2_optimality}
A lens space $(L(p,q),\xi)$ admits a minimal symplectic filling $X$ with
\[
|\pi_1(X)|= d \quad\text{and}\quad b_2(X) =\frac{p}{d^2} -1
\]
if and only if  $L(p,q)$ is homeomorphic to $L(nd^2,ndc-1)$ where $n=\frac{p}{d^2}$; $c$ is an integer satisfying $\gcd(c,d)=1$ and $1\leq c \leq d$; and $\xi$ is universally tight.
\end{prop}

Note that setting $n=1$ in Proposition~\ref{prop:main2_optimality} recovers the classification of lens spaces which bound symplectic rational balls.

Moreover, it turns out that a symplectic filling $X$ of $L(p,q)$ with $|\pi_1(X)|= d$ and $b_2(X) =\frac{p}{d^2} -1$, as in Proposition~\ref{prop:main2_optimality}, is unique up to diffeomorphism (see Remark~\ref{rem:uniqueness}) and plays an interesting role in singularity theory. Every $L(p,q)$ arises as the link of $(X_{p,q},0)$ a cyclic quotient singularity (See \cite{Nemethi} for a nice introduction to these matters). It is shown in \cite[Proposition 5.9]{Looijenga-Wahl1986} that $(X_{p,q},0)$ has a smoothing which is a quotient of a Gorenstein smoothing if and only if the pair $(p,q)$ is of the form described in Proposition~\ref{prop:main2_optimality}. In fact, one can verify that fillings under consideration are precisely those arising as the Milnor fibers of smoothings obtained as quotients of Gorenstein smoothings (again, see Remark~\ref{rem:uniqueness}). 

To state our second bound we need to establish some notation. If $p > q > 0$ are relatively prime integers, then the rational number $p/q$ admits a unique Hirzebruch-Jung continued fraction expansion  $$p/q=[a_1,a_2,\dots, a_k]^-$$ where $a_i$ are integers with $a_i \geq 2$. We use $\len(p/q)=k$ to denote the length of this continued fraction. If $X$ is a minimal symplectic filling of $L(p,q)$ then the following inequality holds \cite[Theorem~1]{Fossati:2019-1}, \cite[Theorem~1.4]{Etnyre-Roy:2020-1}:
$$b_2(X)\leq \len(p/q).$$
Moreover, Fossati showed that this inequality is strict if $\pi_1 (X)$ is non-trivial. We show that in fact the gap between $b_2(X)$ and $\len(p/q)$ grows with the size of the fundamental group $\pi_1 (X)$. Let $F_i$ denote the Fibonacci sequence indexed so that $F_1=F_2=1$.

\begin{thm}\label{thm:main}
If $X$ is a minimal symplectic filling of $(L(p,q), \xi)$ with $|\pi_1(X)| \geq F_{\ell+2}$, then
\begin{equation}\label{eq:b2_lengthbound}
b_2(X)\leq \len(p/q) - \ell.
\end{equation}\end{thm}
 In fact, since the Fibonacci numbers are approximately exponential in the golden ratio $\varphi$, there is an explicit upper bound for $b_2(X)$ of the form:
$$b_2(X)  < \len(p/q) -\log_\varphi |\pi_1(X)| +1.$$
We can also exhibit examples showing that the bound \eqref{eq:b2_lengthbound} is sharp.
\begin{prop}\label{prop:main_optimality}
A lens space $(L(p,q),\xi)$ admits a non-simply connected minimal symplectic filling $X$ with
\[
|\pi_1(X)|= F_{\ell+2} \quad\text{and}\quad b_2(X)= \len(p/q) -\ell
\]
if and only if  $L(p,q)$ is homeomorphic to $L(nF_{\ell+2}^2,nF_{\ell}F_{\ell+2}-1)$ for some integer $n\geq 1$, and $\xi$ is universally tight.
\end{prop}

Lisca gave an explicit construction which yields all symplectic fillings of the standard contact structures on lens spaces \cite{Lisca:2008-1} and Etnyre-Roy \cite{Etnyre-Roy:2020-1} and Christian-Li \cite{Christian-Li:2020-1} showed that in fact any minimal symplectic filling of a lens space is diffeomorphic to one of the fillings constructed by Lisca. Thus we prove our results by analyzing the topology of Lisca's fillings, which have a somewhat intricate description in terms continued fraction expansions and admissable tuples.

\subsection*{Acknowledgments}
The authors would like to thank John Etnyre and Edoardo Fossati for enlightening correspondence.

\section{Continued fractions}\label{sec:propcont}

Consider Hirzebruch-Jung continued fraction expansions:
\[
[a_1,\dots, a_k]^- := a_1 - \cfrac{1}{a_2 - \cfrac{1}{\ddots - \cfrac{1}{a_k}}}.
\]
If $p>q>0$ are relatively prime integers, then the rational number $p/q$ admits a unique expression as a continued fraction in the form
\[
p/q=[a_1, \dots, a_k]^-,
\]
where the $a_i$ are integers with $a_i \geq 2$. Using this expansion we can define several notions of `size' for rational numbers. We define $\len,U,V:\Q_{>1} \rightarrow \Z$ as follows. We define
\[\len(p/q):=k\]
to be the length of the expansion and we define
\[
U(p/q):=\sum_{i=1}^k (a_i-2)
\quad\text{and}\quad
V(p/q):=\sum_{i=1}^k(a_i -1).
\]
Notice that these satisfy
\[
V(p/q)=\len(p/q)+U(p/q).
\]
We will make use of the following properties of continued fractions. These facts are all well-known and are straightforward to prove by induction. For example, proofs of \eqref{it:frac0} and \eqref{it:frac1} can be found in \cite[Proposition~10.2.2]{Cox-Little-Schenck}. Proofs of \eqref{it:reverse_order} and \eqref{it:frac2} can be obtained similarly.

\begin{lem}\label{lem:cont_frac_properties}
Let $\{a_i\}_{i\geq 1}$ be a sequence of integers. Let $p_0=1$, $q_0=0$, $p_1=a_1$, $q_1=1$ and for $k\geq 2$ define $p_k$ and $q_k$ recursively by
\begin{equation}\label{eq:define_convergents}
p_k=a_k p_{k-1} - p_{k-2} \quad\text{and}\quad q_k=a_k q_{k-1} - q_{k-2}.
\end{equation}
Then these satisfy the following properties:
\begin{enumerate}[font=\upshape]
\item\label{it:frac0} $p_k/q_k=[a_1,\dots, a_k]^-$ whenever $q_k\neq 0$,
\item\label{it:reverse_order} $p_k/p_{k-1}=[a_k,\dots, a_1]^-$ whenever $p_{k-1}\neq 0$,
\item\label{it:frac1} $q_k p_{k-1} - p_k q_{k-1}=1$ for all $k\geq 1$, and
\item\label{it:frac2} $[a_1,\dots, a_k,x]^-=\frac{x p_k -p_{k-1}}{x q_k -q_{k-1}}$ for any $x\in \R$ with $x\neq \frac{q_{k-1}}{q_k}$.\qed
\end{enumerate}
\end{lem}
\subsection{The Fibonacci numbers}\label{subsec:fib}
It will be convenient to define the following operations $S,T : \Q_{>1} \rightarrow \Q_{>1}$ on the rationals by setting
\[S(p/q)=\frac{p+q}{q} \quad\text{and}\quad T(p/q)=\frac{2p-q}{p}.\]
These are chosen so that if $a_i$ are integers with $a_i \geq 2$, then
\[
S([a_1, \dots, a_k]^-)=[a_1+1, \dots, a_k]^-
\quad\text{and}\quad
T([a_1, \dots, a_k]^-)=[2,a_1, \dots, a_k]^-.
\]
Since every rational $p/q>1$ has a unique continued fraction expansion with integer coefficients greater than 1, we see that every $p/q>1$ is obtained from $2/1$ by a unique sequence of applications of the $S$ and $T$ operators. Furthermore, note that $p/q>1$ is in the image of the $T$ operator if and only if $p/q<2$ and $p/q$ is in the image of the $S$ operator if and only if $p/q>2$.

Under these definitions we can easily see the following:
\begin{enumerate}[label=(\roman*)]
\item $V(T(p/q))=V(S(p/q))=V(p/q)+1$,
\item $U(S(p/q))=U(p/q)+1 \quad\text{and}\quad U(T(p/q))=U(p/q)$, and
\item $\len(S(p/q))=\len(p/q) \quad\text{and}\quad \len(T(p/q))=\len(p/q)+1.$
\end{enumerate}

So we see that $V(p/q)-1$ is counting the number of applications of $S$ and $T$ required to obtain $p/q$ from $2/1$; that $U(p/q)$ is counting the number of applications of $S$ required to obtain $p/q$ from $2/1$ and that $\len(p/q)-1$ is counting the number of applications of $T$ required to obtain $p/q$ from $2/1$.

\begin{lem}\label{lem:ULidentities}
If $p>q>0$ are relatively prime integers, then we have that
\begin{enumerate}[font=\upshape]
\item\label{it:1ULidentities} $V(p/q)= V(p/(p-q))$,
\item $\len(p/q)=U(p/(p-q))+1$,
\item $\len(p/q)+ \len(p/(p-q)) =V(p/q)+1$, and
\item\label{eq:UUL} $U(p/q)+ U(p/(p-q)) =V(p/q)-1.$
\end{enumerate}
\end{lem}
\begin{proof} We only sketch the proof and the details are left to the reader. We use induction. The base case is $p/q=2/1$ for which the identities are evident. Now suppose that 
\[p/q=S(p'/q')= \frac{p'+q'}{q'}.\]
for some relatively prime integers $p'>q'>0$, then
\[T\left(\frac{p'}{p'-q'}\right) = \frac{p'+q'}{p'}=\frac{p}{p-q}.\]
Similarly (by the symmetry of the situation), if
 \[p/q=T(p'/q') = \frac{2p'-q'}{p'}	\]
for some relatively prime integers $p'>q'>0$, then
 \[
 S\left(\frac{p'}{p'-q'}\right) = \frac{2p'-q'}{p'-q'} = \frac{p}{p-q}.
 \]
This allows us to prove all the identities by observing how $U$, $V$ and $\len$ are changed by applications of the $S$ and $T$ operators.
\end{proof}

Let $F_n$ denote the Fibonacci numbers indexed so that $F_1=F_2=1$.
\begin{lem}\label{lem:fibonacci}
Let $p>q>0$ be relatively prime integers. If $V(p/q)=L$, then
\[p\leq F_{L+2}\]
with equality if and only if $p/q$ takes the form
\[
p/q=F_{L+2}/F_L \quad \text{or} \quad p/q=F_{L+2}/F_{L+1}.
\]
\end{lem}
\begin{proof}
Suppose that we have some $p/q>1$ which maximizes $p$ for a given $L$. Suppose that we can write this as $p/q =G\circ T (r/s)$ where $G$ is some composition of the $S$ and $T$ operators and $r>s>0$ are relatively prime integers. Thus we see that $p/q$ can be written in the form
\[
p/q=[b_1, \dots, b_k, r/s]^-
\]
for some integers $b_i$ with $b_i \geq 2$. Set
\[
p'/q'=\left[b_1, \dots, b_k, \frac{r}{r-s}\right]^-,
\]
where $p'>q'>0$ are relatively prime integers. By Lemma~\ref{lem:cont_frac_properties} \eqref{it:frac2} there are integers $n$ and $m$, such that $p=nr-ms$ and $p'=nr - m(r-s)$. The maximality of $p$ we have that $p'\leq p$ (note that by Lemma~\ref{lem:ULidentities} \eqref{it:1ULidentities} we have that $r/s$ and $r/(r-s)$ have the same $V$ value). This implies that $r/s\geq 2$, which implies that either $r/s=2/1$ or $r/s$ is in the image of the $S$ operator.

Likewise suppose that 
$p/q =G\circ S (r/s)$ where $G$ is some composition of the $S$ and $T$ operators.
Thus we see that $p/q$ can be written in the form
\[
p/q=\left[b_1, \dots, b_k, t + \frac{r}{s}\right]^-
\]
for some integers $b_i$ and $t$ with $b_i\geq 2$ and $t\geq 1$. Set
\[
p'/q'=\left[b_1, \dots, b_k, t+ \frac{r}{r-s}\right]^-.
\]
Similarly, there are integers $n>m$ such that 
\[p=n(r+ts)-ms \quad\text{and}\quad p'=n(r+t(r-s)) - m(r-s).\]
The maximality of $p$ implies that $p'\leq p$. This implies that $r/s\leq 2$, which implies that either $r/s=2/1$ or $r/s$ is in the image of the $T$ operator. Thus we have shown that if $p$ is maximal then $p/q= G (2/1)$ where $G$ is an alternating product of the $S$ and $T$ operators. This implies that a $p/q$ with $p$ maximal takes one of the forms
\[
p/q=(ST)^n(2/1) \quad \text{or}\quad p/q=(TS)^n(2/1)
\]
if $L=2n+1$ is odd or 
\[
p/q=T(ST)^n(2/1) \quad \text{or}\quad p/q=S(TS)^n(2/1)
\]
if $L=2n+2$ is even. In terms of continued fractions these give
\[
p/q=[3,\dots, 3]^- \quad \text{or}\quad p/q=[2,3,\dots, 3,2]^-
\]
or 
\[
p/q=[3,\dots, 3,2]^- \quad \text{or}\quad p/q=[2,3,\dots, 3]^-
\]
depending on whether $L$ is odd or even. It easy to check that these are continued fraction expansions for
\[
p/q=F_{L+2}/F_L \quad \text{and} \quad p/q=F_{L+2}/F_{L+1},
\]
as required.
\end{proof}
\subsection{Matrix identities}
Given a tuple of real numbers $\xtup=(x_1, \dots, x_k)$ we will use $M(\xtup)$ to denote the matrix
\[
M(\xtup)= \begin{pmatrix}
    x_1 & -1    & 0 & 0 \\
    -1 & x_2 & -1 & 0 \\ 
 0  &-1  & \ddots & -1  \\
 0  & 0 & -1    &x_k    
\end{pmatrix}.
\]
If one attempts to calculate the determinant of $M(\xtup)$ by expanding the first row or column one quickly arises at the following recursive formula:
\begin{equation}\label{eq:det_recursion}
\det(M(x_1, \dots, x_k))=x_1 \det (M(x_2, \dots, x_k))- \det (M(x_3,\dots,x_k)),
\end{equation}
valid when $k\geq 3$. By comparing \eqref{eq:det_recursion} to \eqref{eq:define_convergents} in Lemma~\ref{lem:cont_frac_properties} one can easily prove the following lemma by induction. We will leave the proof to the reader.

\begin{lem}\label{lem:cont_frac_matrices}
Let $\{a_i\}_{i\geq 1}$ be a sequence of positive integers. If $k \geq 2$, then the convergents of the continued fraction
\[
p_k/q_k=[a_1, \dots, a_k]^-
\]
where $p_k$ and $q_k$ are positive relatively prime integers, can be computed as 
\[\pushQED{\qed}
p_k=\det M(a_1, \dots, a_k)
\quad\text{and}\quad
q_k=\det M(a_2, \dots, a_k).
\qedhere \]
\end{lem}

The following identity, which is an easy application of the multilinearity of the determinant, will also be useful.
\begin{lem}\label{lem:det_multilinear} Let $\{a_i\}_{i\geq 1}$ be a sequence of positive integers. If $1 < i < k$, then
\begin{align*}\pushQED{\qed}
 \det M(a_1, \dots, a_{i-1}, a_i +m, a_{i+1}, \dots, a_k)&=  
 \det M(a_1, \dots, a_k)\\ &+m \det M(a_1, \dots, a_{i-1})\det M(a_{i+1}, \dots, a_k).\qedhere
\end{align*}
\end{lem}

\subsection{Admissable tuples}
Let $\n=(n_1, \dots, n_k)$ be a tuple of non-negative integers. For each $1\leq i \leq k$, let
\[
\alpha_i/ \beta_i= [n_1, \dots, n_i]^-.
\]
We say that $\n$ is an \emph{admissable tuple} if the matrix $M(\n)$ is positive semi-definite with rank at least $k-1$.


It follows that $\n$ is admissable then for $1\leq i \leq j \leq k$ we have
\begin{equation*}\label{eq:det_positivity}
\det (M(n_i, \dots, n_j))\geq 0
\end{equation*}
with equality only if $i=1$ and $j=k$.

Let $\n=(n_1, \dots, n_k)$ be a tuple with $n_j=1$ for some $1 \leq j \leq k$. We define an operation called \emph{blow down} of $\n$ at $j$ as follows:
\[
\n' =
\begin{cases} 
(n_1, \dots, n_{j-1}-1, n_{j+1}-1, \dots, n_k) &\text{if $1<j<k$,}\\
(n_1, \dots, n_{k-2}, n_{k-1}-1) &\text{if $j=k$,}\\
(n_2-1, n_3, \dots, n_{k}) &\text{if $j=1$.}\\
\end{cases}
\]
It readily follows from the definition of admissibility that $\n$ is admissable if and only if $\n'$ is admissable.\footnote{The point is that one can perform a change of basis that converts the matrix $M(\n)$ into $\begin{pmatrix} 
M(\n') & 0\\
0& 1
\end{pmatrix}$.}

\begin{lem}\label{lem:U_bound}
Suppose that $\n=(n_1, \dots, n_k)$ is an admissable tuple with
\[
\alpha/ \beta =[n_1, \dots, n_k]^-,
\]
where $\alpha$ and $\beta$ are positive relatively prime integers. Then
\begin{equation}\label{eq:U_bound}
\sum_{i=1}^k (n_i -2)\geq U(\alpha/\beta^*)
\end{equation}
where $1\leq \beta^*<\alpha$ and $\beta^* \equiv \beta	\bmod{\alpha}$. Moreover, we have equality only if $n_i\geq 2$ for all $1< i<k$.
\end{lem}
\begin{proof}
Suppose that we $n_j=1$ for some $j$ and let
\[\n'=(n_1', \dots, n_{k-1}')\]
be the admissable tuple obtained by blowing down $\n$ at $j$. 
\begin{equation*}
[n_1', \dots, n_{k-1}']^- = \begin{cases}
\frac{\alpha}{\beta} &\text{if $j>1$}\\
\frac{\alpha}{\beta-\alpha} &\text{if $j=1$}
\end{cases}.
\end{equation*}

However we find that
\begin{equation}\label{eq:blowdown_U}
\sum_{i=1}^{k-1} (n_i'-2) =
\begin{cases}
-1+\sum_{i=1}^{k} (n_i-2) & \text{if $j<k$,}\\
\sum_{i=1}^{k} (n_i-2)  &\text{if $j=1$ or $j=k$.}
\end{cases}
\end{equation}
That is blow-downs cannot increase the sum in the left hand side of \eqref{eq:U_bound}.
Since the operation of blowing down decreases the length of the tuple it cannot be repeated indefinitely. Thus after some finite number of blowdowns the tuple $\n$ will be converted to the tuple
$(c_1, \dots, c_{k'})$ where $c_i\geq 2$ for all $i$ and 
\[\alpha/\beta^* =[c_1, \dots, c_{k'}]^-.\] 
We have
\[
\sum_{i=1}^{k} (n_i-2) \geq \sum_{i=1}^{k'} (c_i-2) = U(\alpha/\beta^*).
\]
The statement about equality comes from observing that if $n_j=1$ for some $1<j<k$, then \eqref{eq:blowdown_U} shows that blowing down strictly decreases the quantity $\sum_{i=1}^{k} (n_i-2)$.
\end{proof}
The following facts concerning admissable tuples will also be useful.
\begin{lem}\label{lem:admissable_facts}
Let $\n=(n_1, \dots, n_k)$ be an admissable tuple such that
\[
[n_1, \dots, n_k]^-=0
\]
and for each $1\leq i \leq k$ let
\[
\alpha_i/\beta_i = [n_1, \dots, n_i]^-,
\]
where $\alpha_i$ and $\beta_i$ are positive relatively prime integers. Then we have the following:
\begin{enumerate}[font=\upshape]
\item\label{admissable1} $[n_i, \dots, n_k]^-=\alpha_{i-2}/\alpha_{i-1}$ for any $1<i \leq k$,
\item\label{admissable2} $\det M(n_1, \dots, n_{i-1})=\det M(n_{i+1}, \dots, n_k)$ for any $1<i<k$, and
\item\label{admissable3}
if $n_i=1$ for some $1<i<k$, then
\[
[n_k, \dots, n_{i+1} ]^- = \frac{\alpha_{i-1}}{m\alpha_{i-1}-\beta_{i-1}}.
\]
for some integer $m\geq 1$.
\end{enumerate}
\end{lem}
\begin{proof}
For~\eqref{admissable1}, suppose that $r/s=[n_i, \dots, n_k]^-$, then by Lemma~\ref{lem:cont_frac_properties} \eqref{it:frac2} we have that
\[
0=[n_1, \dots, n_{i-1}, r/s]^-=\frac{ \alpha_{i-1}(r/s)-\alpha_{i-2}}{\beta_{i-1}(r/s)-\beta_{i-2}}.
\]
This implies that $r/s = \alpha_{i-2}/\alpha_{i-1}$ as required.

For~\eqref{admissable2}, we observe that Lemma~\ref{lem:cont_frac_matrices} combined with \eqref{admissable1} implies that $\alpha_{i-1}$ is computed as a determinant in two ways:
\[
 \det M(n_1, \dots, n_{i-1})=\det M(n_{i+1}, \dots, n_k)=\alpha_{i-1}.
  \]

For~\eqref{admissable3}, suppose that $n_i=1$. This implies that
\[
 \alpha_{i-2}/\alpha_{i-1}= [1,n_{i+1}, \dots, n_k]^- = 1-\frac{1}{[n_{i+1}, \dots, n_{k} ]^-},
\]
and hence that
\[
[n_{i+1}, \dots, n_{k} ]^-=\frac{\alpha_{i-1}}{\alpha_{i-1}-\alpha_{i-2}}.
\]
Now suppose that
\[
[n_{k}, \dots, n_{i+1} ]^-=c/d,
\]
where $c,d>0$ are relatively prime integers. By Lemma~\ref{lem:cont_frac_properties}~\eqref{it:reverse_order} we have that $c=\alpha_{i-1}$. Furthermore, by applying Lemma~\ref{lem:cont_frac_properties}~\eqref{it:frac1} we see that $d$ satisfies
\[
(\alpha_{i-1}-\alpha_{i-2})d \equiv 1 \bmod{\alpha_{i-1}}.
\]
This implies that $d\equiv -\beta_{i-1}\bmod{\alpha_{i-1}}$, since Lemma~\ref{lem:cont_frac_properties}~\eqref{it:frac1} shows that $\beta_{i-1}\alpha_{i-2}\equiv 1 \bmod{\alpha_{i-1}}$.\end{proof}

\section{Symplectic fillings}\label{sec:symplectic}
 
We first recall Lisca's~\cite{Lisca:2008-1} construction of symplectic fillings for lens spaces. Given a lens space $L(p,q)$, consider continued fraction expansion
\[
\frac{p}{p-q}=[b_1, \dots, b_k]^-,
\]
where $b_i$ are integers with $b_i \geq 2$. Let $\n=(n_1, \dots, n_k)$ be an admissible tuple with 
\begin{equation}\label{eq:admissable}
[n_1,\dots, n_k]^-=0 \quad\text{and}\quad b_i\geq n_i \text{ for $i=1, \ldots, k$}.
\end{equation}
Note that this condition implies that  the following surgery diagram describes $S^1 \times S^2$.

\begin{figure}[h]
\includegraphics[width=.9\textwidth]{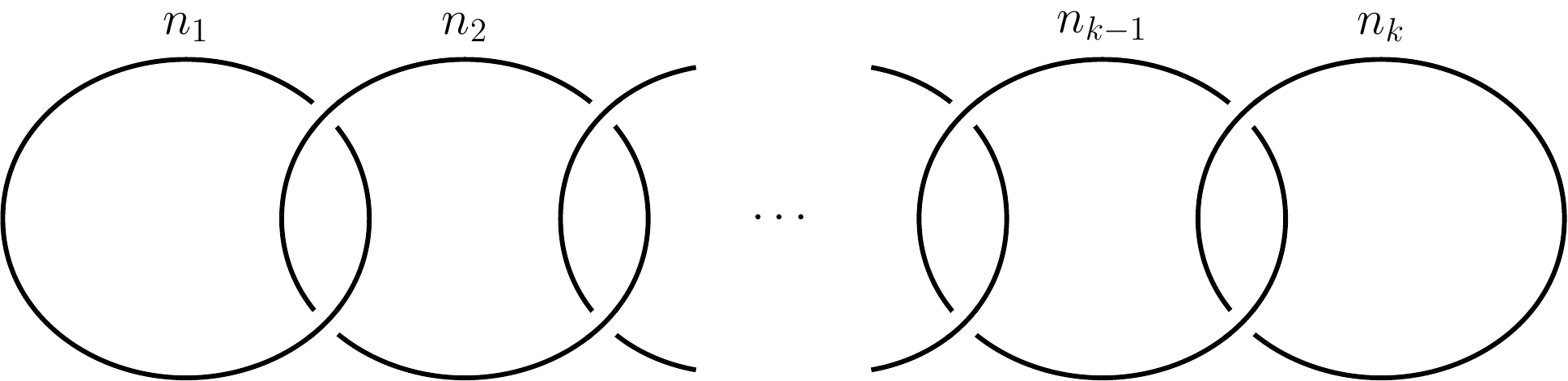}
\caption{A surgery description of $S^1 \times S^2$ corresponding to $\n$}\label{fig:s1s2}
\end{figure}

\noindent Moreover, we construct a 4-manifold $W_{p,q}(\n)$ by attaching $2$-handles to $S^1 \times D^3$ as in Figure~\ref{fig:wn}. 

\begin{figure}[h]
\includegraphics[width=.9\textwidth]{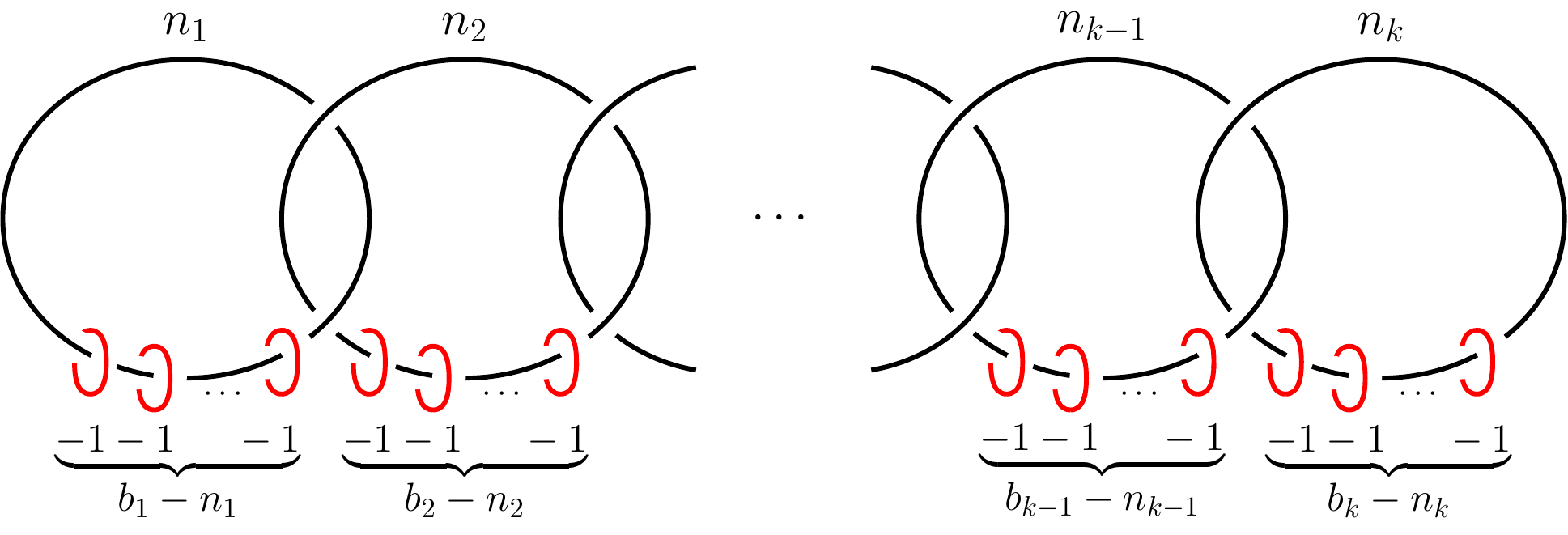}
\caption{The manifold $W_{p,q}(\n)$.}\label{fig:wn}
\end{figure}

\noindent By~\cite[Theorem 1]{Lisca:2008-1}, every minimal symplectic filling of a universally tight contact structures on $L(p,q)$ is diffeomorphic to $W_{p,q}(\n)$ for some $\n$ satisfying \eqref{eq:admissable}. It is straight forward to calculate some invariants of $W_{p,q}(\n)$. First note that by Lemma~\ref{lem:ULidentities}, we have
\[
\len(p/q)=U\left(\frac{p}{p-q}\right)+1 = \sum_{i=1}^k (b_i-2)+1.
\] 
Also, by construction of $W_{p,q}(\n)$ we have that
\begin{equation}\label{eq:b2_formula}
b_2(W_{p,q}(\n))=\sum_{i=1}^k (b_i-n_i)-1.
\end{equation}
Thus we see that
\begin{equation}\label{eq:b2_len_difference}
\len(p/q)- b_2(W_{p,q}(\n))= \sum_{i=1}^k (n_i-2)+2.
\end{equation}
Now we compute the fundamental group. Given an admissable tuble $\n=(n_1, \dots, n_k)$ satisfying \eqref{eq:admissable} and an integer $i$ with $1\leq i \leq k$, let
\[
\alpha_i/ \beta_i= [n_1, \dots, n_i]^-,
\]
where $\alpha_i$ and $\beta_i$ are positive relatively prime integers. Also, we set $\alpha_0 = 1$ and $\beta_0 = 0$.

\begin{lem}\label{lem:pi1_calc}
For each integer $i$ with $1\leq i \leq k$, let $d_i=\alpha_{i-1}$ if $b_i>n_i$ and $d_i=0$ otherwise. If $d=\gcd (d_1, \ldots, d_k)$, then
\[\pi_1(W_{p,q}(\n)) \cong \Z /d\Z.\]
\end{lem}

\begin{proof}
Since $W_{p,q}(\n)$ is obtained by attaching 2-handles to a single 1-handle, it's fundamental group is cyclic and hence abelian. Thus it suffices to compute $H_1(W_{p,q}(\n))$. Let $\mu_i$ denote the homology class of the meridian of the $i$th unknot component of the link in Figure~\ref{fig:s1s2}. Notice that these satisfy the relation
\[
-\mu_{i-1}+n_i \mu_i - \mu_{i+1} =0.
\]
for $1\leq i \leq k$ (with the convention that $\mu_0=\mu_{k+1}=0$). Thus notice that  if $\mu_i=\gamma_i \mu_1$ and $3 \leq i \leq k$, then these satisfy the recursion relation
\[
\gamma_i= n_{i-1} \gamma_{i-1} - \gamma_{i-2} 
\]
where $\gamma_1=1$ and $\gamma_2=n_1$. However we see that the sequence $\alpha_{i-1}$ satisfies the same recursion relation by \eqref{eq:define_convergents} in Lemma~\ref{lem:cont_frac_properties} and we have $\alpha_{0}=1$ and $\alpha_{1}=n_1$. Hence we conclude that $\gamma_i=\alpha_{i-1}$ for each $1\leq i \leq k$. In particular, we have that $\gamma_k=\alpha_{k-1}=1$ since $[n_1, \dots, n_k]^-=0$. Thus we conclude that $H_1(W_{p,q}(\n))$ is generated by $\mu_1=\mu_k$ with a relation $\alpha_{i-1} \mu_1=0$ for each 2-handle attached along $\mu_i$ (i.e.\ when $b_i < n_i$).
\end{proof}

We will later make use of the following consequence.

\begin{cor}\label{cor:divide}
 If $b_i>n_i$ for some $1<i<k$, then $|\pi_1(W_{p,q}(\n))|$ divides
\[
\det(M(n_1, \dots, n_{i-1}))=\det(M(n_{i+1}, \dots, n_{k})).
\]
\end{cor}

\begin{proof} By Lemma~\ref{lem:cont_frac_matrices} and Lemma~\ref{lem:admissable_facts} \eqref{admissable2}, we have \[
\alpha_{i-1}=\det(M(n_1, \dots, n_{i-1}))=\det(M(n_{i+1}, \dots, n_{k})).
\] The proof is complete by applying Lemma~\ref{lem:pi1_calc}.
\end{proof}

\begin{lem}\label{lem:extremal_example}
Suppose $W_{p,q}(\n)$ is a non-simply connected symplectic filling of $L(p,q)$ where there is a unique $j$ such that $b_j>n_j$.
Then $j>1$ and there are coprime integers $c,d$ with $0< c < d$ and  $d/c=[n_1, \dots, n_{j-1}]^-$, such that
\begin{enumerate}[font=\upshape]
\item $L(p,q)$ is diffeomorphic to $L(nd^2, ncd-1)$,
\item $|\pi_1(W_{p,q}(\n))| = d$, and 
\item $\len(p/q)-b_2(W_{p,q}(\n)) = V(d/c)$,
\end{enumerate}
where $b_2(W_{p,q}(\n))= n-1$.\end{lem}

\begin{proof}
Since we are assuming that $W_{p,q}(\n)$ is not simply connected, $\n$ takes the form $[n_1, \dots, n_k]^-=0$, where $k>1$. The tuple $\n$ must have some $n_j=1$. Since $b_j\geq 2$, this implies that this must be the unique index with $n_j=1$. Moreover thus, $\n$ takes the form
\[
\n=(n_1, \dots, n_{j-1}, 1, n_{j+1} , \dots, n_{k}). 
\]
Moreover, as $W_{p,q}(\n)$ is not simply connected, we have that $1<j<k$. Also, the assumption $b_2(W_{p,q}(\n))=n-1$ and equation \eqref{eq:b2_formula} implies that
\[
\frac{p}{p-q}=[n_1, \dots, n_{j-1}, 1+n, n_{j+1} , \dots, n_{k}]^-.
\]
Since $n_1, \dots, n_{j-1}\geq 2$, then there are relatively prime integers $c,d$ such that $d>c>0$ and  
\[
\frac{d}{c}=[n_1, \dots, n_{j-1}]^-
\]
By Lemma~\ref{lem:pi1_calc} we have $\pi_1 (W_{p,q}(\n)) = \Z/d\Z$.
Lemma~\ref{lem:admissable_facts} \eqref{admissable3} combined with the fact that $d>c>0$ shows that
\[
\frac{d}{d-c}=[n_{k} , \dots, n_{i+1}]^-.
\]
Thus by equation~\eqref{eq:b2_len_difference} and Lemma~\ref{lem:ULidentities} \eqref{eq:UUL}, we see that
\[
\len(p/q)- b_2(W_{p,q}(\n))= 2 + \sum_{i=1}^k (n_i-2)=1+U\left(\frac{d}{c}\right)+U\left(\frac{d}{d-c}\right)=V\left(\frac{d}{c}\right).
\]
It is an exercise in continued fractions to calculate $p/(p-q)$ as 
\[
\frac{p}{p-q}= \frac{nd^2}{ndc+1}
\]
Thus, $p/q$ takes the form:
 \[\frac{p}{q}= \frac{nd^2}{nd^2- ndc -1}.\]
 However since $(nd^2+ ndc -1)(ndc-1)\equiv 1 \bmod nd^2$, we see that $L(p,q)$ is homeomorphic to $L(nd^2, ndc-1)$, as required. 
\end{proof}

\begin{lem}\label{lem:univ_tight}
Let $X$ be a minimal symplectic filling of $(L(nd^2, q), \xi)$ where $q\equiv -1 \bmod nd$ and $|\pi_1 (X) |=d$. Then $\xi$ is universally tight.\end{lem}
\begin{proof}
By deforming the symplectic structure on $X$, we can assume that it is a Stein filling of $(L(p,q),\xi)$ \cite{Wendl-Niederkruger}. We have a surjection $\pi_1 (L(p,q)) \rightarrow \pi_1 (X)$ induced by inclusion. Thus taking the universal cover of $X$ yields a Stein filling $\widetilde{X}$ of the lens space $(L(p',q'), \widetilde{\xi})$, where $L(p',q')\cong L(nd, q') \cong L(nd, nd-1)$.  Since $\widetilde{X}$ is Stein, the contact structure $\widetilde{\xi}$ is tight \cite{Gromov1985}. The classification of tight contact structures on lens space shows that every tight contact structure on $L(nd, nd-1)$ is universally tight \cite{Honda2000}. Thus we have universally tight $\widetilde{\xi}$ covering $\xi$. This implies that $\xi$ is itself universally tight.\end{proof}

\subsection{Proof of Theorem~\ref{thm:main} and Proposition~\ref{prop:main_optimality}}
\begin{proof}[Proof of Theorem~\ref{thm:main} and Proposition~\ref{prop:main_optimality}]
By \cite[Theorem 1.1]{Etnyre-Roy:2020-1} and \cite[Theorem 1.4]{Christian-Li:2020-1} any minimal symplectic filling of a lens space is diffeomorphic to one constructed by Lisca~\cite{Lisca:2008-1} as described at the begining of Section~\ref{sec:symplectic}. So suppose that $X\cong W_{p,q}(\n)$ is a filling of $L(p,q)$ corresponding to an admissable tuple
\[
\n = (n_1, \dots, n_k),
\]
satisfying \eqref{eq:admissable}. 

Note that if $X$ simply connected, the theorem is equivalent to the assertion that $\len(p/q) - b_2(X)\geq 0$. This was established Fossati~\cite[Theorem 1]{Fossati:2019-1}. Alternatively the enthusiastic reader can deduce it from \eqref{eq:b2_len_difference} and properties of admissable tuples\footnote{\textbf{Hint:} The admissable tuple $\n$ can be reduced to the admissable tuple $(0)$ by a sequence of blow-downs.}. 

In any case, we may assume that $X$ is not simply connected. In particular, by Lemma~\ref{lem:pi1_calc} we can assume that $n_1,n_k\geq 2$. Now the admissable tuple must have $n_j=1$ for some $1<j<k$. By picking $j$ minimal with this property we can assume that the admissible tuple takes the form
\[
\n=(n_1, \dots, n_{j-1}, 1, n_{j+1}, \dots, n_k),
\]
where we can assume that $n_1,\dots, n_{j-1}\geq 2$. 
Suppose that
\[
[n_{1}, \dots, n_{j-1}]^-=\frac{\alpha}{\beta}>1,
\]
where $\alpha$ and $\beta$ are positive relatively prime integers. Then by Lemma~\ref{lem:admissable_facts}~\eqref{admissable3}, we have that
\[
[n_{k}, \dots, n_{j+1}]^-=\frac{\alpha}{m\alpha-\beta}.
\]
for some integer $m\geq 1$. Since $\alpha>\beta$, Lemma~\ref{lem:U_bound} implies that
\begin{equation}\label{eq:proof_main_sum}
\sum_{i=j+1}^k (n_i-2) \geq  U\left(\frac{\alpha}{\alpha-\beta}\right)
\end{equation}
Also, Corollary~\ref{cor:divide} implies that 
\begin{equation}\label{eq:pi1leqalpha}
|\pi_1(X)|\leq \alpha.
\end{equation}
Moreover, by equation~\eqref{eq:b2_len_difference}, Lemma~\ref{lem:ULidentities} \eqref{eq:UUL}, and equation~\eqref{eq:proof_main_sum}, we have
\begin{equation}\label{eq:proof_main1}
\len(p/q)-b_2(X)\geq U\left(\frac{\alpha}{\beta}\right)+U\left(\frac{\alpha}{\alpha-\beta}\right)+1 = V\left(\frac{\alpha}{\beta}\right).
\end{equation}
Thus if we write $L=V(\alpha/\beta)$, then  equation~\eqref{eq:proof_main1} along with equation~\eqref{eq:pi1leqalpha} and Lemma~\ref{lem:fibonacci} show that
\[
b_2(X)\leq \len(p/q)- L \quad\text{and}\quad | \pi_1(X)|\leq F_{L+2}.
\]
Thus if $| \pi_1(X)|\geq F_{\ell+2}$ for some integer $\ell \geq 0$, we have that $L\geq \ell$ and the bound \eqref{eq:b2_lengthbound} follows.
 
Next we deduce Proposition~\ref{prop:main_optimality}. We continue to work with $X$ non-simply connected and the same notation as before. Suppose that we have
\[
|\pi_1(X)|=F_{\ell+2} \quad\text{and}\quad b_2(X)= \len(p/q)- \ell  
\]
for some $\ell\geq 1$. This implies that $|\pi_1(X)|=\alpha= F_{\ell+2}$ and $\ell=V(\alpha/\beta)$. By Lemma~\ref{lem:fibonacci} this implies that $\alpha/\beta	=F_{L+2}/F_L$ or $\alpha/\beta= F_{L+2}/F_{L+1}$. Lemma~\ref{lem:U_bound} implies that if
\[
\len(p/q)-b_2(X) =2+\sum_{i=1}^k (n_i-2) = V(\alpha/\beta)
\]
then $n_{j+1}, \dots, n_k\geq 2$. Thus $\n$ has the unique index $j$ with $n_j=1$. Moreover, since $n_1, \dots, n_{j-1}\geq 2$, the sequence of denominators $\alpha_i$ is increasing for $i<j$ and so for  $i<j$ we have $\alpha_{i-1}<\alpha_{j-1}$. Thus if $b_i>n_i$ for some $i< j$, we would have $|\pi_1 (X)|<\alpha$. By reversing the  order of the admissable tuple we can show similarly that if $i>j$ satisfies $b_i>n_i$ then we would have $|\pi_1 (X)|<\alpha$. Thus we see that $j$ is the unique index with $b_j>n_j$. Thus it follows from Lemma~\ref{lem:extremal_example} and Lemma~\ref{lem:univ_tight} that $X$ is a filling of $L(n F_{\ell +2}^2, n F_{\ell+2} F_\ell -1)$ or $L(n F_{\ell +2}^2, n F_{\ell+2} F_{\ell+1} -1)$ with a universally tight contact structure. However since $F_{\ell+2} =F_\ell + F_{\ell +1}$ these two lens spaces are diffeomorphic. This proves the only if statement. Conversely, guided by Lemma~\ref{lem:extremal_example}, it is easy to produce an admissable tuple $\n$ for which $W_{p,q}(\n)$ has all the necessary properties.\end{proof}

\subsection{Proof of Theorem~\ref{thm:main2} and Proposition~\ref{prop:main2_optimality}}
\begin{proof}[Proof of Theorem~\ref{thm:main2} and Proposition~\ref{prop:main2_optimality}]
First we verify Theorem~\ref{thm:main2} and Proposition~\ref{prop:main2_optimality} in the case that $d=1$. Let $X$ be a minimal simply connected filling of $(L(p,q), \xi)$. Since $b_2(X)\leq \len(p/q)\leq q \leq p-1$, the bound \eqref{eq:b2-bound} is satisfied in this case. Moreover, we have equality implies that $q=p-1$. Lemma~\ref{lem:univ_tight} shows that if $q=p-1$, then $\xi$ must be universally tight. This proves the moreover statement in Theorem~\ref{thm:main2} and establishes the only if direction of Proposition~\ref{prop:main2_optimality}. To complete the proof of Proposition~\ref{prop:main2_optimality}, note that a standard plumbing fills $L(p,p-1)$ and has all the necessary properties.

Thus from now now on we can assume that $X$ is a non-simply connected minimal symplectic filling of $(L(p,q),\xi)$. Then \cite[Theorem 1.1]{Etnyre-Roy:2020-1} and \cite[Theorem 1.4]{Christian-Li:2020-1} implies $X$ is diffeomorphic to some $W_{p,q}(\n)$ as constructed by Lisca~\cite{Lisca:2008-1}. So suppose that $X\cong W_{p,q}(\n)$ is a filling of $L(p,q)$ corresponding to an admissable tuple
\[
\n = (n_1, \dots, n_k),
\]
satisfying \eqref{eq:admissable}, and
\[\frac{p}{p-q}=[b_1, \dots, b_k]^-.\] 
By Lemma~\ref{lem:cont_frac_matrices} We have that
\[
p=\det M(b_1, \dots, b_k).
\]
Let $S\subseteq \{1,\dots, k\}$ be the set of $i$ with $b_i>n_i$. As we are assuming that $X$ is not simply-connected and Lemma~\ref{lem:pi1_calc} shows that $1$ and $k$ are not in $S$. By using multilinearity of the determinant as in Lemma~\ref{lem:det_multilinear} and inducting on the size of the set $S$, we see that $p$ can be expressed here as 
\begin{align*}
p&= \sum_{i\in S} (b_i-n_i) \det M(n_1, \dots, n_{i-1}) \det M(n_{i+1}, \dots, n_{k})\\
&+ \sum_{\substack{i,j\in S \\ i<j}} K_{i,j} \det M(n_1, \dots, n_{i-1}) \det M(n_{j+1}, \dots, n_{k}),
\end{align*}
where $K_{i,j}$ is the integer
\[
K_{i,j}= (b_i-n_i)(b_j-n_j) \det M(b_{i+1}, \dots, b_{j-1}).
\]
 In any case we see that for any $i\in S$ we have that $d$ divides both 
$\det M(n_{1}\dots , n_{i-1})$ and $\det M(n_{i+1}\dots , n_{k})$. So the above expression implies that $d^2$ divides $p$. Morever the property that the tuple is admissable implies that  the coefficients satisfy $K_{i,j}\geq 1$. Thus by ignoring terms that correspond to subsets of $S$ with more than one element we obtain the bound
\begin{equation}\label{eq:d2_bound}
p\geq d^2\sum_{i\in S} (b_i-n_i)= d^2(b_2(X)-1),
\end{equation}
with equality only if $S$ contains a single element. Rearranging the above inequality yields \eqref{eq:b2-bound}.

Now suppose that $X$ is a filling with $b_2(X)=p/d^2-1$ and $|\pi_1(X)|=d$. Since equality in \eqref{eq:d2_bound} implies that the admissible tuple $\n$ has a unique index $j$ with $b_j>n_j$, we see that $X$ is a filling of the form studied in Lemma~\ref{lem:extremal_example}. This shows that $L(p,q)$ is diffeomorphic to $L(nd^2, ndc-1)$, where $n=b_2(X)+1$. Moreover, Lemma~\ref{lem:univ_tight} shows that the contact structure $\xi$ is universally tight. This shows that the filling is in the form required by Proposition~\ref{prop:main2_optimality} and establishes the moreover statement in Theorem~\ref{thm:main2}.

Guided by Lemma~\ref{lem:extremal_example}, one can easily find admissable tuples necessary to construct fillings completing the proof of Proposition~\ref{prop:main2_optimality}. 
\end{proof}

\begin{rem}\label{rem:uniqueness}

Lemma~\ref{lem:extremal_example} and the proof of Proposition~\ref{prop:main2_optimality} imply that $L(p,q)$ has a unique minimal symplectic filling $X$ with with $|\pi_1(X)|= d$ and $b_2(X) =\frac{p}{d^2} -1$. It is established in the proof Proposition~\ref{prop:main2_optimality} that this $X$ is diffeomorphic to $W_{p,q}(\n)$, where there is exactly one index $j$ such that $b_j>n_j$. It follows from Lemma~\ref{lem:extremal_example} the admissable tuple $\n$ is determined by the fraction $p/q$ and the integer $d$. Hence, by \cite[Theorem~1.1]{Lisca:2008-1} we conclude that the filling $X$ is unique up to diffeomorphism.
In particular, since $X$ is unique it must coincide with the filling obtained as the Milnor fiber of a quotient of a Gorenstein smoothing (cf. \cite[Proposition 5.9]{Looijenga-Wahl1986}).
\end{rem}

\bibliographystyle{alpha}
\def\MR#1{}
\bibliography{bib}
\end{document}